\documentclass[12pt]{amsart}

\usepackage{amsmath}
\usepackage{amssymb}
\usepackage{caption}
\usepackage{subcaption}
\usepackage{float}
\usepackage[curve]{xypic}
\usepackage{cite}
\usepackage{enumitem}
\usepackage{color}
\usepackage{url}
\usepackage[usenames,dvipsnames]{xcolor}
\usepackage{graphicx}
\usepackage{verbatim}
\usepackage{tikz}
\usepackage[normalem]{ulem}
\usepackage{hyperref}
\hypersetup{
     colorlinks   = true,
     citecolor    = black,
     linkcolor    = blue
}

\makeatletter 
\def\@cite#1#2{{\m@th\upshape\bfseries%
[{#1\if@tempswa{\m@th\upshape\mdseries, #2}\fi}]}}
\makeatother 

\theoremstyle{plain}
\newtheorem{thm}{Theorem}[section]

\newtheorem{prop}[thm]{Proposition}

\newtheorem{lemma}[thm]{Lemma}

\theoremstyle{definition}
\newtheorem{defn}[thm]{Definition}

\newtheorem{ass}[thm]{Assumption}

\theoremstyle{}
\newtheorem{rem}[thm]{Remark}

\numberwithin{equation}{section}
\renewcommand{\bold}[1]{\medskip \noindent {\bf #1 }\nopagebreak}

\newcommand{\nc}{\newcommand}
\newcommand{\rnc}{\renewcommand}

\nc\bA{\mathbb{A}}
\nc\bB{\mathbb{B}}
\nc\bC{\mathbb{C}}
\nc\bD{\mathbb{D}}
\nc\bE{\mathbb{E}}
\nc\bF{\mathbb{F}}
\nc\bG{\mathbb{G}}
\nc\bH{\mathbb{H}}
\nc\bI{\mathbb{I}}
\nc{\bJ}{\mathbb{J}} 
\nc\bK{\mathbb{K}}
\nc\bL{\mathbb{L}}
\nc\bM{\mathbb{M}}
\nc\bN{\mathbb{N}}
\nc\bO{\mathbb{O}}
\nc\bP{\mathbb{P}}
\nc\bQ{\mathbb{Q}}
\nc\bR{\mathbb{R}}
\nc\bS{\mathbb{S}}
\nc\bT{\mathbb{T}}
\nc\bU{\mathbb{U}}
\nc\bV{\mathbb{V}}
\nc\bW{\mathbb{W}}
\nc\bY{\mathbb{Y}}
\nc\bX{\mathbb{X}}
\nc\bZ{\mathbb{Z}}
\nc\cA{\mathcal{A}}
\nc\cB{\mathcal{B}}
\nc\cC{\mathcal{C}}
\rnc\cD{\mathcal{D}}
\nc\cE{\mathcal{E}}
\nc\cF{\mathcal{F}}
\nc\cG{\mathcal{G}}
\rnc\cH{\mathcal{H}}
\nc\cI{\mathcal{I}}
\nc{\cJ}{\mathcal{J}} 
\nc\cK{\mathcal{K}}
\rnc\cL{\mathcal{L}}
\nc\cM{\mathcal{M}}
\nc\cN{\mathcal{N}}
\nc\cO{\mathcal{O}}
\nc\cP{\mathcal{P}}
\nc\cQ{\mathcal{Q}}
\rnc\cR{\mathcal{R}}
\nc\cS{\mathcal{S}}
\nc\cT{\mathcal{T}}
\nc\cU{\mathcal{U}}
\nc\cV{\mathcal{V}}
\nc\cW{\mathcal{W}}
\nc\cY{\mathcal{Y}}
\nc\cX{\mathcal{X}}
\nc\cZ{\mathcal{Z}}
\nc\wt{\widetilde}

\newcommand{\ra}{\longrightarrow}
\newcommand{\M}{\mathcal{M}}

\newcommand{\C}{\mathbb C}
\newcommand{\R}{\mathbb R}

\newcommand{\Q}{\mathbb Q}

\newcommand{\GL}{\mathrm{GL}}

\nc{\dmo}{\DeclareMathOperator}
\rnc{\Re}{\operatorname{Re}}
\rnc{\Im}{\operatorname{Im}}
\rnc{\span}{\operatorname{span}}
\dmo{\rank}{rank}
\dmo{\End}{End}
\dmo{\Hom}{Hom}
\dmo{\Jac}{Jac}
\dmo{\Id}{Id}
\dmo{\Ann}{Ann}
\dmo{\Area}{Area}
\dmo{\CP}{\bC P^1}
\dmo{\Aut}{Aut}

\title{Rank One Orbit Closures in $\mathcal{H}^{hyp}(g-1,g-1)$}

\author[Apisa]{Paul~Apisa}

\begin{document}
\maketitle

\begin{abstract}
All $\GL(2, \R)$ orbits in hyperelliptic components of strata of abelian differentials in genus greater than two are closed, dense, or contained in a locus of branched covers.
\end{abstract}

\section{Introduction}

Teichm\"uller geodesic flow induces a $\GL(2, \R)$ action on the moduli space of holomorphic one-forms $\Omega \M_g$. The space $\Omega \M_g$ admits a $\GL(2, \R)$-invariant stratification by specifying the number and order of zeros of the holomorphic one-forms. The components of these strata were classified by Kontsevich and Zorich~\cite{KZ} and shown to be distinguished by hyperellipticity and spin parity. The closure of a $\GL(2, \R)$ orbit is a linear submanifold by work of Eskin-Mirzakhani~\cite{EM} and Eskin-Mirzakhani-Mohammadi~\cite{EMM}. In Apisa~\cite{Apisa-hyp}, it was shown that in hyperelliptic components of strata of $\Omega \M_g$ orbit closures of dimension greater than three are either loci of branched covers or the entire component of the stratum. In the sequel, we strengthen this theorem to show the following:

\begin{thm}\label{T1}
In hyperelliptic components of strata of abelian differentials in genus greater than two every $\GL(2, \R)$ orbit is closed, dense, or contained in a locus of branched covers.
\end{thm}

Theorem~\ref{T1} and Eskin-Filip-Wright~\cite[Theorem 1.5]{EFW} imply that in these components there are only finitely many closed orbits that are not contained in a locus of branched covers. All invariant measures and orbit closures in hyperelliptic components are known up to computing this finite collection of Teichm\"uller curves. 

The proof uses the explicit geometry of abelian differentials in $\mathcal{H}^{hyp}(g-1,g-1)$ (developed in Lindsey~\cite{Lindsey} and extended in Apisa~\cite{Apisa-hyp}) and cylinder deformation results of Wright~\cite{Wcyl} to reduce the problem to the classification of orbit closures in genus two, which was achieved by McMullen in~\cite{Mc},~\cite{McM:spin},~\cite{Mc4}, and~\cite{Mc5}. 

\bold{Acknowledgements.} The author thanks Alex Eskin and Alex Wright for useful conversations and encouragement.  This material is based upon work supported by the National Science Foundation Graduate Research Fellowship Program under Grant No. DGE-1144082. The author gratefully acknowledges their support.

\section{Background}\label{S:Background}

Throughout this section we make the following assumption

\begin{ass}\label{A1}
Let $(X, \omega)$ be a horizontally periodic translation surface in an affine invariant submanifold $\M$. Suppose that $(C_i)_{i=1}^n$ is the collection of horizontal cylinders and that $C_i$ has height $h_i$, modulus $m_i$, and core curve oriented form left to right $\gamma_i$ of length $c_i$ for $i = 1, \hdots, n$. If a cylinder $C$ is specified without an index, then these quantities will be denoted $h_C$, $m_C$, $\gamma_C$, and $c_C$ respectively. Let $W \subseteq \Q^n$ be the collection of rational homogeneous linear relations satisfied by $m = (m_i)_{i=1}^n$, i.e. a vector $w$ belongs to $W$ if and only if $w \cdot m = 0$. 
\end{ass}

\begin{lemma}[Wright~\cite{Wcyl}, Corollary 3.4]\label{L:Wcyl2}
If $v \in \R^n$ belongs to $W^\perp$ then 
\[ \sum_{i=1}^n c_i v_i \gamma_i^* \in T_{(X, \omega)} \M \]
where $\gamma_i^*$ is the cohomology class dual to $\gamma_i$. 
\end{lemma}

\begin{defn}
The twist space of a horizontally periodic translation $(X, \omega)$ in an affine invariant submanifold $\M$ is the complex linear span of all tangent vectors of $T_{(X, \omega)} \M$ of the form $\sum_{i=1}^n a_i \gamma_i^*$. The cylinder preserving space is the subspace of $T_{(X, \omega)} \M$ consisting of cohomology classes that evaluate to zero on the core curves of all horizontal cylinders. An element of this space is called a cylinder preserving deformation. 
\end{defn}

The tangent space to $\M$ at $(X, \omega)$ is a subspace of the relative cohomology group $H^1(X, \Sigma; \C)$ where $\Sigma$ is the collection of zeros of $\omega$. Let $p: H^1(X, \Sigma; \C) \ra H^1(X; \C)$ be the projection from relative to absolute cohomology. The image of $T_{(X, \omega)} \M$ under $p$ is always complex-symplectic by Avila-Eskin-M\"oller~\cite{AEM}.

\begin{lemma}[Wright~\cite{Wcyl}, proof of Theorem 1.10]\label{L:Wcyl1} 
For any affine invariant submanifold $\M$ there is a horizontally periodic translation surface $(X, \omega) \in \M$ whose twist space and cylinder preserving space coincide. On such a surface, the twist space contains all of $\ker p \cap T_{(X, \omega)} \M $ and projects to a Lagrangian subspace of $p\left(T_{(X, \omega)} \M \right)$.
\end{lemma}

\begin{defn}
The rank of an affine invariant submanifold is defined to be half the complex dimension of $p \left( T_{(X, \omega)} \M \right)$ for any $(X, \omega) \in \M$. This definition is independent of the choice of $(X, \omega)$. The rel is defined to be the complex dimension of  the intersection of $\ker p$ and $T_{(X, \omega)} \M$. An element of this intersection is called a rel deformation. 
\end{defn}. 

\section{Sub-equivalence Classes in $\mathcal{H}^{hyp}(g-1,g-1)$}

Throughout this section Assumption~\ref{A1} will remain in effect. Moreover, we will make the following assumption.

\begin{ass}\label{A2}
$\M$ is a rank one rel one nonarithmetic orbit closure in $\mathcal{H}^{hyp}(g-1, g-1)$ for some $g > 2$. 
\end{ass}

\begin{defn}\label{D:eta}
By Apisa~\cite[Theorem 6.3]{Apisa-hyp}, if $(X, \omega)$ is a horizontally periodic translation surface whose twist space contains a rel deformation, then the rel deformation is given (up to scaling) by 
\[ \eta := i \cdot \sum_{j=1}^n q_j \gamma_j^* \qquad \text{where} \qquad q_j = (-1)^{d(C_j, C_1)}\]
where $d(C_j, C_1)$ is the distance between $C_j$ and $C_1$ in the cylinder graph - the graph whose vertices are horizontal cylinders and edge relations correspond to cylinder adjacencies. Notice that $\eta$ is characterized by the property that one of the coefficients (and hence all coefficients) of $\{\gamma_j^* \}_{j=1}^n$ belongs to $\{ \pm i \}$. Hence $\eta$ is defined up to multiplication by $\pm 1$. Since $\eta$ is a tangent vector on a linear submanifold, we may move in the $\eta$ direction, and this alters the heights of the horizontal cylinders. When a cylinder reaches height zero, it is said to collapse. 
\end{defn}

\begin{lemma}\label{L:facts}
There is a horizontally periodic translation surface $(X, \omega)$ in $\M$ that contains $\eta$ in its twist space. By moving an arbitrarily small distance in the $\eta$ direction it is possible to ensure that the modulus of any horizontal cylinder that expands under $\eta$ is not a rational multiple of the modulus of any horizontal cylinder that contracts under $\eta$.
\end{lemma}
\begin{proof}
By Lemma~\ref{L:Wcyl1}, there is a translation surface $(X, \omega)$ whose twist space contains the rel deformation $\eta$. Moving in the $\eta$ direction causes some horizontal cylinders to expand and others to contract. The collection of times in which one contracting and one expanding horizontal cylinder have a rational ratio of moduli is countable. Therefore, moving in the $\eta$ direction for an arbitrarily small amount of time that lies outside of this countable set produces the desired surface.
\end{proof}

\begin{lemma}\label{L:facts2}
If $(X, \omega)$ contains $\eta$ in its twist space and not all its horizontal cylinders have a rational ratio of moduli, then 
\[ W^\perp = \span_\C \left\{ m, (q_i/c_i)_{i=1}^n \right\} \]
and the vector space of rational homogeneous linear relations satisfied by $(1/c_i)_{i=1}^n$ is exactly $W \cdot \mathrm{diag}(q_i)$.
\end{lemma}
\begin{proof}
Recall that $W$ is the vector space of rational homogeneous linear relations satisfied by the moduli $m$ of the horizontal cylinders. By Lemma~\ref{L:Wcyl2}, each element of $W^\perp$ corresponds to an element of the twist space, which is at most two-dimensional since $\M$ is rank one rel one. Therefore, $W^\perp$ is no more than two dimensional. Since not all horizontal cylinders have a rational ratio of moduli, $W$ is at least, and hence exactly, codimension two. Since $W^\perp$ contains $m$ by definition and since $\eta$ belongs to the twist space by assumption we see that $W^\perp$ is spanned by $m$ and $(q_i/c_i)_{i=1}^n$.

By Assumption~\ref{A2}, $\M$ is nonarithmetic and so, by Wright~\cite[Theorem 1.9]{Wcyl}, the $\Q$-span of $\{1/c_i\}_{i=1}^n$ in $\R$ has dimension at least two. Therefore, the subspace $U$ of rational homogeneous linear relations satisfied by $(1/c_i)_{i=1}^n$ is at least codimension two. Since $(1/c_i)_{i=1}^n$ satisfies all the rational linear relations in the codimension two subspace $W \cdot \mathrm{diag}(q_i)$ we see that $U$ coincides with this space.  
\end{proof}

\begin{rem}
The second statement of Lemma~\ref{L:facts2} is one of the key observations of this paper. It shows that the rational homogeneous linear relations that the moduli of horizontal cylinders satisfy can be recovered from the relations satisfied by the reciprocals of the lengths of core curves (and vice versa). 
\end{rem}

\begin{lemma}\label{L:arithmetic}
It is not the case that all contracting horizontal cylinders in both the $+\eta$ and $-\eta$ directions collapse simultaneously.
\end{lemma}
\begin{proof}
Suppose not; then all contracting (resp. expanding) horizontal cylinders have identical heights, i.e. there are two disjoint groups of horizontal cylinders, $\cA$ and $\cB$, so that $\sum_{a \in \cA} \gamma_a^*$ and $\sum_{b \in \cB} \gamma_b^*$ span the tangent space. This implies that all cylinders in a given group have a constant, and hence rational, ratio of moduli. Since cylinders in a given group have identical heights, any two such cylinders have a rational ratio of lengths of core curves. The linear relation on homology classes of core curves given by $\eta$ now implies that all core curves have a rational ratio of lengths and hence, by Wright~\cite[Theorem 1.9]{Wcyl}, that $\M$ is a rank one arithmetic orbit closure, which contradicts Assumption~\ref{A2}. 
\end{proof}

\begin{defn}
Two cylinders on a translation surface in $\M$ will be said to be sub-equivalent if they are equivalent and they have a generically rational ratio of moduli. A maximal collection of sub-equivalent cylinders will be called a sub-equivalence class. 
\end{defn}

\begin{prop}\label{P:3sc}
If $(X, \omega)$ contains a rel deformation in its twist space, then there are three sub-equivalence classes and sub-equivalent horizontal cylinders have identical heights.
\end{prop}

\begin{rem}
If Theorem~\ref{T1} holds, then there is a map from $(X, \omega)$ to a horizontally periodic translation surface in $\mathcal{H}(1,1)$ with three horizontal cylinders. A sub-equivalence class in $(X, \omega)$ is then just a collection of horizontal cylinders that map to a given cylinder on the surface in $\mathcal{H}(1,1)$. Proposition~\ref{P:3sc} must hold if Theorem~\ref{T1} does. 
\end{rem}

\begin{proof}
If necessary, by Lemma~\ref{L:facts}, we move an arbitrarily small distance in the $\eta$ direction to guarantee that no horizontal cylinder that expands in the $\eta$ direction has a modulus that is a rational multiple of one that contracts (this allows us to apply Lemma~\ref{L:facts2}). By Lemma~\ref{L:arithmetic}, we may assume (perhaps after multiplying $\eta$ by $-1$) that in the $\eta$ direction, not all contracting horizontal cylinders simultaneously collapse. Let $\cA_1$ be those that collapse first and let $\cA_2$ be the remaining horizontal cylinders that contract in the $\eta$ direction. Let $\cA_3$ be the horizontal cylinders that expand in the $\eta$ direction. Finally, suppose that the cylinders in $\cA_2 \cup \cA_3$ are the ones labelled $\left(C_1, \hdots, C_k \right)$. 

Notice that a cylinder in $\cA_i$ and a cylinder in $\cA_j$ are sub-equivalent only if $i \ne j$. We will show conversely that $\cA_i$ is a sub-equivalence class for each $i = 1, 2, 3$.

Suppose without loss of generality, perhaps after applying $\begin{pmatrix} 1 & t \\ 0 & 1 \end{pmatrix}$ for arbitrarily small $t$, that when $\cA_1$ collapses, no vertical saddle connection vanishes. Let $(Y, \zeta)$ be the translation surface formed when $\cA_1$ collapses (by moving in the $\eta$ direction). This surface is horizontally periodic and the horizontal cylinders are precisely the ones belonging to $\cA_2 \cup \cA_3$. Since $\M$ is nonarithmetic, Wright~\cite[Theorem 1.9]{Wcyl} implies that not all elements of $(c_i)_{i=1}^k$ are rational multiples of each other.  Let $U \subseteq \Q^k$ be the subset of rational homogeneous linear relations that $(1/c_i)_{i=1}^k$ satisfy. By Lemma~\ref{L:facts2} applied to $(X, \omega)$, the $\Q$-linear span of $\{ 1/c_i \}_{i=1}^n$ in $\R$ is two-dimensional, and so $U$ is a codimension two subspace. 

Notice that while $\eta$ no longer lies in the twist space of $(Y, \zeta)$, it still corresponds to a tangent vector in $T_{(Y, \zeta)} \M$ and so we may continue to move in the $\eta$ direction. Doing so creates a horizontally periodic translation surface $(X', \omega')$ that contains $\eta$ in its twist space (notice that the surface is horizontally periodic because it contains a horizontal cylinder and any such surface in $\M$ is horizontally periodic by Wright~\cite[Theorem 1.5]{Wcyl} since $\M$ is rank one). 

On $(X', \omega')$ the $\eta$ direction remains well-defined, although it becomes a different linear combination of the core curves of horizontal cylinders. Since the coefficients of $\eta$ are purely imaginary on $(X, \omega)$, they remain purely imaginary on $(X', \omega')$ and so the $\eta$ direction on $(X', \omega')$ agrees with the direction defined in Definition~\ref{D:eta}. As at the start of this proof we may apply Lemma~\ref{L:facts} to assume that no horizontal cylinder that expands in the $\eta$ direction has a modulus that is a rational multiple of one that contracts (and so Lemma~\ref{L:facts2} holds on $(X', \omega')$). 

To visualize these definitions in the case of a surface in $\mathcal{H}(1,1)$ see Figure~\ref{F:mnemonic}; in that figure, cylinder $C_i$ belongs to $\cA_i$ for $i = 1, 2, 3$, $a$ and $b$ are used to indicate side identifications (along with the convention that unlabelled opposite sides are identified). The $\eta$ direction is written as a linear combination of core curves whenever possible with $\gamma_1'$ denoting the core curve of the cylinder $C_1'$.

\begin{figure}[H]
    \begin{subfigure}[b]{0.3\textwidth}
        \centering
        \resizebox{\linewidth}{!}{\begin{tikzpicture}
		\draw (0,0) -- (0,1) -- (.5, 1.5) -- (1.5, 1.5) -- (1,1) -- (2,1) -- (2,-1.5) -- (1,-1.5) -- (1,0) -- (0,0);
		\draw[dotted] (0,1) -- (1,1) -- (1,0) -- (2,0);
		\node at (.7, 1.25) {$C_1$}; \node at (1, .5) {$C_3$}; \node at (1.5, -.75) {$C_2$};
		\node at (-1, .5) {$$}; \node at (3, .5) {$$};
	\end{tikzpicture}
        }
        \caption{$(X, \omega)$ \\ $-i\eta = \gamma_3^* - \gamma_1^* - \gamma_2^* $}
    \end{subfigure}
        \begin{subfigure}[b]{0.3\textwidth}
        \centering
        \resizebox{\linewidth}{!}{\begin{tikzpicture}
        		\draw (0,0) -- (0,1.5) -- (2,1.5) -- (2,-1) -- (1,-1) -- (1,0) -- (0,0);
		\draw[dotted] (1,1.5) -- (1,0) -- (2,0);
		\draw[black, fill] (.5,1.5) circle[radius=1pt]; 
		\draw[black, fill] (.5,0) circle[radius=1pt]; 
		\node at (.25, 1.7) {$a$}; \node at (.75, 1.7) {$b$};
		\node at (.25, -.2) {$b$}; \node at (.75, -.2) {$a$}; 
		\node at (1, .75) {$C_3$}; \node at (1.5, -.5) {$C_2$};
		\node at (-1, .5) {$$}; \node at (3, .5) {$$};
	\end{tikzpicture}
        }
        \caption{$(Y, \zeta)$}
    \end{subfigure}
        \begin{subfigure}[b]{0.3\textwidth}
        \centering
        \resizebox{\linewidth}{!}{\begin{tikzpicture}
        		\draw (0,0) -- (0,1.75) -- (.5, 1.25) -- (1, 1.75) -- (2,1.75) -- (2,-.5) -- (1,-.5) -- (1,0) -- (.5, .5) -- (0,0);
		\draw[dotted] (0, 1.25) -- (2, 1.25); 
		\draw[dotted] (0, .5) -- (2, .5);
		\draw[dotted] (1, 0) -- (2,0);
		\node at (.25, 1.75) {$a$}; \node at (.75, 1.75) {$b$};
		\node at (.25, 0) {$b$}; \node at (.75, 0) {$a$};
		\node at (.5, .87) {$C_3$}; \node at (1.5, -.25) {$C_2$}; \node at (1.5, 1.5) {$C_1'$};
		\node at (-1, .5) {$$}; \node at (3, .5) {$$};
	\end{tikzpicture}
        }
        \caption{$(X', \omega')$ \\ $-i\eta = \gamma_1'^* - \gamma_3^* - \gamma_2^* $}
    \end{subfigure}
\caption{An illustration of the preceding discussion} 
\label{F:mnemonic}
\end{figure}
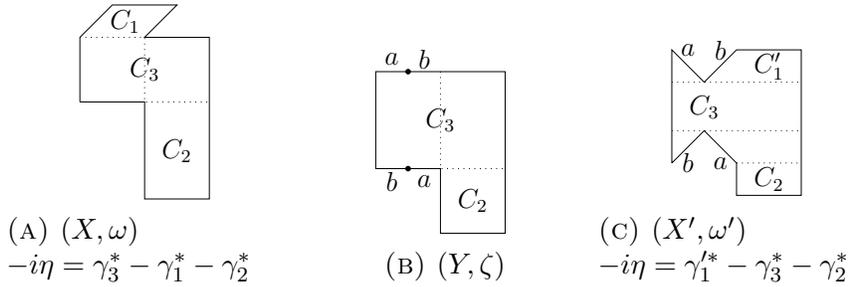

Recall that $(q_i)_{i=1}^k \in \{ \pm 1 \}_{i=1}^k$ is chosen so that $-i \eta$ restricted to the cylinder in $\cA_2 \cup \cA_3$ on $(X, \omega)$ is $\sum_{i=1}^k q_i \gamma_i^*$. Define $(q_i')_{i=1}^k$ analogously for $(X', \omega')$. Let $V \subseteq \Q^k$ be the subspace of rational homogeneous linear equations that the moduli of cylinders in $\cA_2 \cup \cA_3$ satisfy on $(X, \omega)$ and let $V'$ be the analogous subspace for $(X', \omega')$. By Lemma~\ref{L:facts2}, 
\begin{equation}\label{E1}
V \cdot \mathrm{diag}\left( q_1, \hdots, q_k \right) = U = V' \cdot \mathrm{diag}\left( q_1', \hdots, q_k' \right)
\end{equation}
Define $D = \mathrm{diag}\left( q_1 q_1', \hdots, q_k q_k' \right)$, which is a diagonal matrix with all diagonal entries belonging to $\{\pm 1\}$. Notice that a cylinder in $\cA_2 \cup \cA_3$ has $D$-eigenvalue $+1$ if it is contracted (or expanded) in the $\eta$ direction on both $(X, \omega)$ and $(X', \omega')$. If a cylinder is contracted in the $\eta$ direction on one surface, but not the other then its $D$-eigenvalue is $-1$.

\noindent \textbf{Step 1: The $+1$-eigenspace of $D$ contains $\cA_2$ and the $-1$-eigenspace of $D$ is nonempty}


By Definition~\ref{D:eta}, no two horizontal cylinders that contract in the $\eta$ direction are adjacent. In particular, no cylinder in $\cA_2$ is adjacent to a cylinder in $\cA_1$. Since the cylinders in $\cA_1$ are exactly those that collapse on $(Y, \zeta)$, it follows that if $\gamma$ is a cross curve of a cylinder in $\cA_2$ on $(X, \omega)$, then it remains one on $(X', \omega')$. Since the period of $\gamma$ is linear in the $\eta$ direction, any cylinder in $\cA_2$ belongs to the $+1$-eigenspace of $D$.

When a cylinder $C_1$ in $\cA_1$ collapses, it contains a cross curve $s$ - the imaginary part of whose period is decreasing in the $\eta$ direction - that remains a saddle connection on $(Y, \zeta)$ on the boundary of a horizontal cylinder $C_3$. Since cylinders in $\cA_1$ can only border cylinders in $\cA_3$, $C_3$ belongs to $\cA_3$. Since the period of $s$ is linear in the $\eta$ direction, moving in the $\eta$ direction from $(Y, \zeta)$ causes the imaginary part of the period of $s$ to continue to decrease and therefore to form a new cylinder $C_1'$ adjacent to $C_3$ (see Figure~\ref{F:mnemonic}). Since cylinders that expand in the $\eta$ direction, like $C_1'$, can only border cylinders that contract in the $\eta$ direction, $C_3$ contracts in the $\eta$ direction on $(X', \omega')$ and so it belongs to the $-1$-eigenspace of $D$.

\noindent \textbf{Step 2: If two cylinders have the same $D$-eigenvalue, then they have a rational ratio of moduli}

Suppose to a contradiction that two cylinders with $D$-eigenvalue $+1$ do not have a rational ratio of moduli and suppose without loss of generality that their moduli are $m_1$ and $m_2$. Since $V$ is codimension two, all moduli of horizontal cylinders are a rational linear combination of $m_1$ and $m_2$. Suppose without loss of generality that $C_3$ belongs to the $-1$ eigenspace of $D$ and that 
\[ m_3 = \alpha_1 m_1 + \alpha_2 m_2 \qquad \text{for } \alpha_1, \alpha_2 \in \Q \]
is an equation that holds in $V$. By Equation~\ref{E1}, 
\[ -m_3 = \alpha_1 m_1 + \alpha_2 m_2 \]
is an equation that holds in $V'$. Since both equations must hold on $(Y, \zeta)$ it follows that $m_3 = 0$ on $(Y, \zeta)$, which contradicts the fact that $C_3$ is a positive-area cylinder there. The same argument holds for two cylinders of $D$-eigenvalue $-1$.

\noindent \textbf{Step 3: Any two cylinders in $\cA_i$, for $i = 1, 2, 3$, have identical heights and a rational ratio of moduli}

Since cylinders in $\cA_2$ belong to the $+1$-eigenspace of $D$ and since any other such horizontal cylinder must have a modulus that is a rational multiple of those of cylinders in $\cA_2$, it follows that the $+1$-eigenspace of $D$ cannot contain any cylinder in $\cA_3$. Therefore, the cylinders in $\cA_i$, for $i = 2, 3$ have a rational ratio of moduli by the previous step.  The same holds for the cylinders in $\cA_1$ since they have identical heights and hence a generically constant, hence rational, ratio of moduli.

Since any two cylinders in $\cA_1$ have identical heights and since any two cylinders in $\cA_2$, or $\cA_3$, have a rational ratio of moduli,
\[ i\eta = \sum_{c \in \cA_1 \cup \cA_2} \gamma_c^* - \sum_{c \in \cA_3} \gamma_c^* = \sum_{a \in \cA_1} \gamma_a^* + \alpha \sum_{b \in \cA_2} h_b \gamma_b^* - \beta \sum_{c \in \cA_3} h_c \gamma_c^*  \]
where $\alpha$ and $\beta$ are constants. Therefore, any two cylinders in $\cA_i$ have identical heights for $i = 1, 2, 3$. 
\end{proof}

\begin{rem}
Proposition~\ref{P:3sc} immediately implies that the field of definition of $\M$ is quadratic. The proof is omitted since a stronger result will be established in the following section. 
\end{rem} 

\section{Proof of Theorem~\ref{T1}}

By Apisa~\cite{Apisa-hyp}, Theorem~\ref{T1} holds for orbit closures of complex dimension greater than three. Since closed orbits are exactly those whose orbit closures are dimension two, it suffices to consider orbit closures of complex dimension three. Those that are arithmetic are loci of branched covers of tori. Therefore, throughout this section, Assumption~\ref{A2} remains in effect. Using Proposition~\ref{P:3sc}, we will make the following assumption. 

\begin{ass}
Fix a horizontally periodic translation surface $(X, \omega)$ that belongs to $\M$ and that has the relative deformation $\eta$ in its twist space. Let $\cA_i$, for $i = 1, 2, 3$, denote the three sub-equivalence classes and suppose that a cylinder in $\cA_i$ has height $h_i$. Suppose that the relative deformation is 
\[ \sum_{c \in \cA_1 \cup \cA_2} \gamma_c^* - \sum_{c \in \cA_3} \gamma_c^* \]
We will say, abusing notation, that a horizontal saddle connection belongs to $\cA_i$ for $i = 1, 2$ if it borders a cylinder in $\cA_i$. Notice that every saddle connection belongs to exactly one of $\cA_1$ or $\cA_2$.
\end{ass}

We will repeatedly use the following construction in the sequel. 

\begin{defn}
A horizontal saddle connection $s$ on $(X, \omega)$ borders two cylinders, say $C_1$ and $C_2$. We will say that $s$ may be put in standard position if a cylinder preserving deformation may be applied to $(X, \omega)$ so that there is a vertical cylinder $V$ that is contained in $C_1 \cup C_2$, that only intersects the core curves of $C_1$ and $C_2$ once, and that contains $s$ as a cross curve, i.e. $s$ is contained in $V$ and intersects the core curve of $V$ exactly once.
\end{defn}

\begin{lemma}\label{L:stp}
Any horizontal saddle connection on $(X, \omega)$ may be put in standard position.
\end{lemma}
\begin{proof}
Suppose that $C_1$ and $C_2$ are the two horizontal cylinders bordering a horizontal saddle connection $s$ and suppose furthermore that $C_1$ belongs to $\cA_1 \cup \cA_2$ and that $C_2$ belongs to $\cA_3$. Shear all horizontal cylinders so that $s$ and its image under the hyperelliptic involution lie directly above one another in $C_2$. Notice that the tangent space contains
\[ \sum_{c \in \cA_1 \cup \cA_2} \gamma_c^* - \sum_{c \in \cA_3} \gamma_c^* + \sum_{i=1}^3 \sum_{c \in \cA_i} \frac{h_i}{h_3} \gamma_c^* = \sum_{c \in \cA_1 \cup \cA_2} \left( 1 + \frac{h_c}{h_3} \right) \gamma_c^* \]
This deformation fixes all cylinders in $\cA_3$ while shearing those in $\cA_1 \cup \cA_2$. This allows $s$ and its image under the hyperelliptic involution to be put directly above each other in $C_1$, putting $s$ in standard position. 
\end{proof}

\begin{prop}\label{P:main}
For $i = 1, 2$, all saddle connections in $\cA_i$ have identical lengths and if any one is put in standard position, all occur as cross curves of vertical cylinders that only pass through saddle connections in $\cA_i$.
\end{prop}
\begin{rem}
This proposition is the main result of this section. Theorem~\ref{T1} follows directly from it via the arguments in Apisa~\cite[Sections 11 and 12]{Apisa-hyp}. Those arguments will be presented in the following theorem for completeness. 
\end{rem}
\begin{proof}
Since $\M$ is rank one it is completely periodic by Wright~\cite[Theorem 1.5]{Wcyl} and so whenever a saddle connection is put in standard position, the translation surface is vertically periodic. 

\noindent \textbf{Step 1: When a horizontal saddle connection is put in standard position, there are two vertical sub-equivalence classes}

Let $s$ be a horizontal saddle connection and put it in standard position. Let $V$ be the resulting vertical cylinder. 

Suppose first to a contradiction that there is only one sub-equivalence class in the vertical direction. By Proposition~\ref{P:3sc}, every vertical cylinder has identical height and so all horizontal core curves of cylinders have a rational ratio of length. By Wright~\cite[Theorem 1.9]{Wcyl}, $\M$ is arithmetic, contradicting Assumption~\ref{A2}.

Suppose now that there are at least two vertical sub-equivalence classes. Apply the rel deformation $\epsilon \left( \sum_{c \in \cA_1 \cup \cA_2} \gamma_c^* - \sum_{c \in \cA_3} \gamma_c^* \right)$. This creates a gap next to $V$ that can only be entered by a cylinder of height at most $\epsilon$ (see Figure~\ref{F:sp}).
\begin{figure}[H]
    \begin{subfigure}[b]{0.3\textwidth}
        \centering
        \resizebox{\linewidth}{!}{\begin{tikzpicture}
        		\draw (0,0) -- (2,0);
		\draw (0,1) -- (1,1);
		\draw (1,2) -- (3,2);
		\draw (2,1) -- (3,1);
		\draw[dotted] (1,0) -- (1,2);
		\draw[dotted] (2,0) -- (2,2);
		\draw[dotted] (1,1) -- (2,1);
		\node at (1.5, 2.25) {$s$}; \node at (1.5, -.25) {$s$};
		\node at (1.5, 1) {$V$};
	\end{tikzpicture}
        }
        \caption{ Standard position}
    \end{subfigure}
    \qquad
        \begin{subfigure}[b]{0.3\textwidth}
        \centering
        \resizebox{\linewidth}{!}{\begin{tikzpicture}
        		\draw (-.3,0) -- (1.7,0);
		\draw (0,1) -- (1,1);
		\draw (.7,2) -- (3,2);
		\draw (2,1) -- (3,1);
		\draw[dotted] (.7,0) -- (1,1) -- (.7, 2);
		\draw[dotted] (1.7,0) -- (2,1) -- (1.7, 2);
		\draw[dotted] (1,1) -- (2,1);
		\draw[dashed] (1,0) -- (1,2); \draw[dashed] (1.7, 0) -- (1.7, 2);
		\node at (1.2, 2.25) {$s$}; \node at (1.5, -.25) {$s$};
		\node at (1.35, 1) {$V$};
		\draw[black, fill] (.7,0) circle[radius=1pt];
		\node at (.87, -.2) {$\epsilon$};
	\end{tikzpicture}
        }
        \caption{The gap created by $\eta$}
    \end{subfigure}
\caption{The gap created using a rel deformation} 
\label{F:sp}
\end{figure}
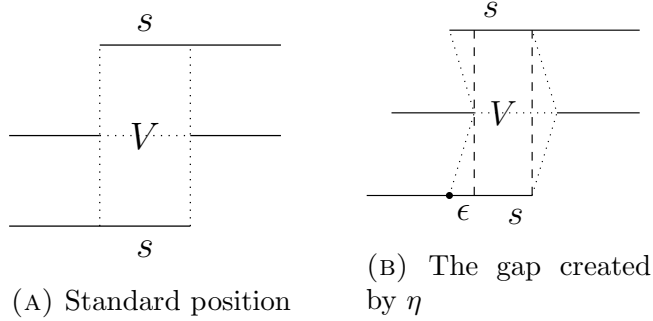
\noindent For $\epsilon$ smaller than $h/2$ where $h$ is the smallest height of a vertical cylinder on $(X, \omega)$ we see that a new equivalence class must pass through this gap and hence by Proposition~\ref{P:3sc} that on $(X, \omega)$ there were only two sub-equivalence classes. 

\noindent \textbf{Step 2: There are two numbers $\ell$ and $x$ so that all horizontal saddle connections on $(X, \omega)$ have length $\ell$ or $nx$ for some integer $n$}

Let $\ell$ be the length of the longest horizontal saddle connection. Let $t$ be a horizontal saddle connection of length $\ell$ on $(X, \omega)$. By Lemma~\ref{L:stp}, we may replace $(X, \omega)$ with a surface on which $t$ is in standard position. Let $V$ be the vertical cylinder containing $t$ as a cross curve.  Every cylinder sub-equivalent to $V$ has height $\ell$ and therefore contains every saddle connection it passes through as a height. 

Let $x$ be the height of the sub-equivalence class of vertical cylinders not containing $V$. Since any saddle connection not contained as a cross-curve of a cylinder sub-equivalent to $V$ is crossed finitely many times by cylinders of height $x$ it follows that all other saddle connections have length $nx$ for some integer $n$. 

Since the ratios of lengths of the core curves generate a field that properly contains $\Q$, we see that $\ell$ and $x$ are not rational multiples of each other. Therefore, when one saddle connection of length $\ell$ is put into standard position, all saddle connections of length $\ell$ occur as cross curves of vertical cylinders.

\noindent \textbf{Step 3: All horizontal saddle connections have length $\ell$ or $x$}

Let $u$ be the longest horizontal saddle connection of length $Nx$ for some integer $N$. Put $u$ into standard position and let $W$ be the resulting vertical cylinder. 

Suppose first that some cylinder sub-equivalent to $W$ passes through a saddle connection of length $\ell$. Since $\ell$ is not an integer multiple of $x$, it follows that cylinders sub-equivalent to $W$ account for all but $\ell - mx$ of the horizontal saddle connection for some integer $m$. Therefore, the second vertical sub-equivalence class has height $\frac{\ell - mx}{n}$ for some integer $n$. None of these cylinders can pass through a saddle connection of length $kx$, where $k$ is an integer, since the only solution to 
\[ kx = a\left( \frac{\ell - mx}{n}\right) + b mx \]
with $a$ and $b$ integers occurs when $a = 0$ (because the $\Q$-span of $\ell$ and $x$ in $\R$ is two-dimensional). Therefore, all horizontal saddle connections of length $kx$ are only intersected by cylinders equivalent to $W$. Since all these cylinders have identical heights - $Nx$ - and since $Nx$ is the length of the longest horizontal saddle connection that has length of the form $nx$ for some integer $n$ - it follows that all horizontal saddle connections that do not have length $\ell$ have length $Nx$. If necessary, set $x$ to $Nx$ so that all horizontal saddle connections have length $\ell$ or $x$.

We see too from the argument that whenever a saddle connection of length $x$ is put into standard position, all other saddle connections of length $x$ occur as cross curves of vertical cylinders. 

Suppose now that cylinders sub-equivalent to $W$ do not pass through saddle connections of length $\ell$. If such a cylinder passes through a saddle connection, then it contains that saddle connection as a cross curve. The other sub-equivalence class of vertical cylinders therefore has cylinders of height $\frac{\ell}{m}$ for some integer $m$. Since no saddle connection of the length $nx$ for some integer $n$ can be contained in finitely many vertical cylinders of height $\frac{\ell}{m}$ we see that every saddle connection of length $nx$ for some integer $n$ is contained in a cylinder sub-equivalent to $W$ and hence has length $Nx$ and occurs in a vertical cylinder as a cross curve. We are now done as before. 

\noindent \textbf{Step 4: All saddle connections in $\cA_i$, for $i = 1, 2$, have identical lengths}

Suppose now, without loss of generality, that $t$ borders a cylinder in $\cA_1$ and $\cA_3$. Therefore, the cylinder $V$, which is formed when $t$ is put in standard position, has height $\ell$ and length $h_1 + h_3$. All sub-equivalent cylinders have a length that is a rational multiple of this length and which must be of the form $a\left( h_1 + h_3 \right) + b\left( h_2 + h_3 \right)$ where $a$ and $b$ are integers. Since the ratio of lengths of core curves in the vertical direction must span a field that properly contains $\Q$  we see that $a\left( h_1 + h_3 \right) + b\left( h_2 + h_3 \right)$ is a rational multiple of $h_1 + h_3$ if and only if $b = 0$. This shows that all cylinders sub-equivalent to $V$ only pass through sub-equivalence classes $\cA_1$ and $\cA_3$. Therefore, all horizontal saddle connections of length $\ell$ border $\cA_1$ and $\cA_3$. The same argument shows that when $u$ is put in standard position, all cylinders sub-equivalent to $W$ pass exclusively through $\cA_2$ and $\cA_3$ and hence that all saddle connections of length $x$ border $\cA_2$ and $\cA_3$.
\end{proof}

The following immediately implies Theorem~\ref{T1}.

\begin{thm}
$\M$ is a branched covering construction of an eigenform locus in $\mathcal{H}(1,1)$.
\end{thm}
\begin{proof}
Suppose to a contradiction that there is a cylinder $C$ in $\cA_3$ on whose boundary the saddle connections in $\cA_1$ and $\cA_2$ do not appear in alternating order. Suppose without loss of generality that the saddle connections $s_1, t_1, t_2$ appear on the top boundary, where $s_1$ belongs to $\cA_1$ and $t_i$ to $\cA_2$ for $i = 1,2$. On the bottom of the cylinder the saddle connections appear in the order $t_2', t_1', s_1'$ where a prime denotes the image of a saddle connection under the hyperelliptic involution. Put $t_1$ into standard position and let $V$ be the resulting vertical cylinder. By Proposition~\ref{P:main}, every cylinder sub-equivalent to $V$ must contain all horizontal saddle connections in $\cA_2$ and must only pass through cylinders in $\cA_2 \cup \cA_3$. However, since a vertical line from some point in $t_2'$ will intersect $s_1$ and hence pass into a cylinder in $\cA_1$, this is not the case and we have a contradiction.

Suppose without loss of generality that the longest saddle connection (length $\ell$) belongs to $\cA_1$ and put it in standard position. By Proposition~\ref{P:main}, all cylinders in $\cA_i$ are horizontally tiled by $C_i$ for $i =1,3$, where $C_i$ for $i = 1, 2, 3$ will refer to the cylinders in Figure~\ref{F:iso} with labels corresponding to lengths, all angles right angles, and $x$ the length of the saddle connections in $\cA_2$. These tilings induce local isometries from cylinders in $\cA_i$ to $C_i$, for $i = 1, 3$, that are undefined at zeros and that take horizontal saddle connections to horizontal saddle connections.

\begin{figure}[H]
    \begin{subfigure}[b]{0.3\textwidth}
        \centering
        \resizebox{\linewidth}{!}{\begin{tikzpicture}
        		\draw (0,0) -- (2,0) -- (2,1) -- (0,1) -- (0,0);
		\node at (-.2, .5) {$h_1$};
		\node at (1, 1.2) {$\ell$};
		\draw[black, fill] (0,0) circle[radius=1pt]; \draw[black, fill] (0,1) circle[radius=1pt];
		\draw[black, fill] (2,0) circle[radius=1pt]; \draw[black, fill] (2,1) circle[radius=1pt];
	\end{tikzpicture}
        }
        \caption{$C_i$: The cylinder by which all cylinders in $\cA_i$ are tiled for $i = 1,2$}
    \end{subfigure}
    \qquad 
        \begin{subfigure}[b]{0.3\textwidth}
        \centering
        \resizebox{\linewidth}{!}{\begin{tikzpicture}
        		\draw (0,0) -- (2,0) -- (2,1) -- (0,1) -- (0,0);
		\draw[dotted] (1,0) -- (1,1);
		\draw[black, fill] (1,0) circle[radius=1pt]; \draw[black, fill] (1,1) circle[radius=1pt];
		\draw[black, fill] (0,0) circle[radius=1pt]; \draw[black, fill] (0,1) circle[radius=1pt];
		\draw[black, fill] (2,0) circle[radius=1pt]; \draw[black, fill] (2,1) circle[radius=1pt];
		\node at (-.2, .5) {$h_3$};
		\node at (.5, 1.2) {$\ell$};
		\node at (1.5, 1.2) {$x$};
	\end{tikzpicture}
        }
        \caption{$C_3$: The cylinder by which all cylinders in $\cA_3$ are tiled}
    \end{subfigure}
\caption{The cylinders that tile those in $\cA_i$ for $i = 1, 2, 3$} 
\label{F:iso}
\end{figure}
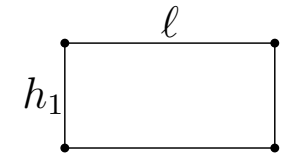
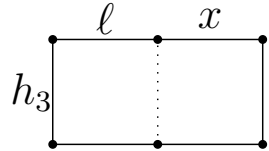

Apply an element of the cylinder preserving space to put a saddle connection of length $x$ into standard position. As in Lemma~\ref{L:stp} we may assume that this deformation acts as the identity on $\cA_3$ and by the matrix $g = \begin{pmatrix} 1 & \alpha \\ 0 & 1 \end{pmatrix}$ on $\cA_1$ for some $\alpha \in \R$. By Proposition~\ref{P:main}, all cylinders in $\cA_2$ are tiled by $C_2$ and admit a local isometry to $C_2$ defined away from zeros (see Figure~\ref{F:iso}). 

Combining these local isometries we have a map $f: (X, \omega) \ra (Y, \zeta)$ where $(Y, \zeta)$ is the translation surface that consists of $g \cdot C_1, C_2, C_3$ glued together as follows (numbered edges indicate gluing).

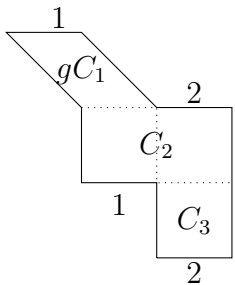
\begin{figure}[H]
\begin{tikzpicture}
	\draw (2,0) -- (2,1) -- (1,1) -- (1,2) -- (0, 3) -- (1, 3) -- (2, 2) -- (3,2) -- (3,0) -- (2,0);
	\draw[dotted] (1,2) -- (2,2) -- (2,1) -- (3,1);
	\node at (1, 2.5) {$g C_1$}; \node at (2, 1.5) {$C_2$}; \node at (2.5, .5) {$C_3$};
	\node at (.7, 3.2) {$1$}; \node at (1.5, .7) {$1$}; \node at (2.5, 2.2) {$2$}; \node at (2.5, -.2) {$2$};
\end{tikzpicture}
\caption{The translation surface $(Y, \zeta)$} 
\label{F:covered-surface}
\end{figure}

Since $(X, \omega)$ may be taken to have dense orbit in $\M$, by \cite[Corollary 4.3]{Apisa-hyp}, it follows that $\M$ is a locus of branched covers of translation surfaces in $\mathcal{H}(1,1)$. Since $\M$ is nonarithmetic and has complex dimension $3$, it must be a locus of branched covers of nonarithmetic eigenforms in $\mathcal{H}(1,1)$, since these are the only such orbit closures of complex dimension $3$ in $\mathcal{H}(1,1)$ by McMullen~\cite{Mc5}.
\end{proof}


\bibliography{mybib}{}
\bibliographystyle{amsalpha}
\end{document}